\documentclass[11 pt]{amsart}

\usepackage{hyperref}
\usepackage{amssymb, amsmath, amsthm, amsfonts}
\usepackage{verbatim}
\usepackage{bbm}
\usepackage{enumerate}
\usepackage{dsfont}
\usepackage{upgreek}
\usepackage[mathscr]{eucal}
\usepackage{tikz}
\usepackage[normalem]{ulem}

\usepackage[backgroundcolor=black!5,linecolor=black]{todonotes}
\usepackage{comment}

\usepackage[explicit,indentafter]{titlesec}
\titleformat{\section}{\centering\normalfont\scshape}{\thesection.}{.5em}{#1}
\titleformat{\subsection}[runin]{\normalfont\itshape}{\textnormal{\thesubsection.}}{.5em}{#1.}
\titleformat{\subsubsection}[runin]{\normalfont\itshape}{\thesubsubsection.}{.5em}{#1.}
\titlespacing{\section}{0em}{1em}{0.5em}
\titlespacing{\subsection}{0em}{.5em}{0.5em}

\usepackage{stackengine}

\usepackage{color}
\definecolor{gray}{gray}{0.5}

\newcommand{\cmt}[1]{}

\newcommand{\vertiii}[1]{{\left\vert\kern-0.25ex\left\vert\kern-0.25ex\left\vert #1 
    \right\vert\kern-0.25ex\right\vert\kern-0.25ex\right\vert}}
    
    \newcommand{\detail}[1]{}

\def\lc{\lesssim}
\def\gc{\gtrsim}

\def\eps{\varepsilon}
\def\bbone{{\mathbbm 1}}

\newcommand{\floor}[1]{\lfloor #1 \rfloor }

\newcommand{\Be}{\begin{equation}}
\newcommand{\Ee}{\end{equation}}

\newcommand{\Bm}{\begin{multline}}
\newcommand{\Em}{\end{multline}}

\def\intslash{\rlap{\kern  .32em $\mspace {.5mu}\backslash$ }\int}
\def\qsl{{\rlap{\kern  .32em $\mspace {.5mu}\backslash$ }\int_{Q_x}}}

\def\F{\mathcal F}

\def\lc{\lesssim}
\def\gc{\gtrsim}

\def\floor#1{{\lfloor #1 \rfloor }}
\def\emph#1{{\it #1 }}

\def\ga{\gamma}

\def\supp{{\mathrm{supp}}}
\def\rad{{\mathrm{rad}}}

\def\inn#1#2{\langle#1,#2\rangle}

\def\ga{\gamma}             
\def\eps{\varepsilon}

\def\ka{\kappa}
             \def\La{\Lambda}

\def\om{\omega}              \def\Om{\Omega}

\def\fA{{\mathfrak {A}}}

\def\fV{{\mathfrak {V}}}

\def\fn{{\mathfrak {n}}}

\def\bbR{{\mathbb {R}}}

\def\bbT{{\mathbb {T}}}

\def\bbZ{{\mathbb {Z}}}

\def\cB{{\mathcal {B}}}

\def\cF{{\mathcal {F}}}

\def\cK{{\mathcal {K}}}

\def\cR{{\mathcal {R}}}
\def\cS{{\mathcal {S}}}
\def\cT{{\mathcal {T}}}
\def\cU{{\mathcal {U}}}

\def\emph#1{{\it #1}}
\def\textbf#1{{\bf #1}}

\def\beq{\begin{equation}}
\def\endeq{\end{equation}}

\def\bs{\begin{split}}
\def\es{\end{split}}

\theoremstyle{plain}

\newtheorem{thm}{Theorem}[section]
\newtheorem{prop}[thm]{Proposition}

\newtheorem{lem}[thm]{Lemma}
\newtheorem{cor}[thm]{Corollary}

\newtheorem*{thm*}{Theorem}
\newtheorem*{conj*}{Conjecture}
\newtheorem*{openproblem*}{Open Problem}

\theoremstyle{remark}

\newtheorem*{remarka}{Remark}

\numberwithin{equation}{section}

\definecolor{ascol}{rgb}{0,0,1.} 
\definecolor{tccol}{rgb}{1.,0,0} 

\begin{document}
\title
{Mollified  disc multipliers}
\author{Anthony Carbery and Andreas Seeger}

\address{Anthony Carbery, School of Mathematics and Maxwell Institute for Mathematical
Sciences, University of Edinburgh, James Clerk Maxwell Building, Peter Guthrie Tait
Rd, Kings Buildings, Edinburgh EH9 3FD, Scotland}

\email{A.Carbery@ed.ac.uk}

\address{Andreas Seeger: Department of Mathematics, University of Wisconsin, 480 Lincoln Drive, Madison, WI, 53706, USA.}
\email{seeger@math.wisc.edu}

\thanks{Research supported in part by NSF grant 2348797.}    

\begin{abstract} We prove a sharp boundedness result on $L^4(\mathbb R^2)$ for mollifications of the disc multiplier.
\end{abstract}




\maketitle 

\section{Introduction} 
Define the Fourier transform of a function $f\in L^1(\bbR^2)$ by 
$\cF f(\xi)\equiv \widehat f(\xi)=\int f(y) e^{- i \inn y\xi} dy.$  For $1\le p\le \infty$  let 
 $M_p\equiv M_p(\bbR^2)$ be the usual Fourier multiplier space with norm 
\[\|m\|_{M_p}= \sup_{\substack{f\in \cS\\\|f\|_p\le 1}} \big\|\cF^{-1}[m\widehat f]\big\|_p. \]  We often use the duality  $\|m\|_{M_p}=\|m\|_{M_{p'}}$. 

Consider the characteristic function $\bbone_D$ of the disc 
$D=\{\xi\in \bbR^2:|\xi|\le 1\}$. A  famous result by C. Fefferman \cite{Fefferman-ball} says that $\bbone_D\in M_p$ if and only if $p=2$.   
Therefore, by an application of the uniform boundedness principle and transference results (\cite[Ch. VII]{stein-weiss}) one can show that $L^p$-convergence for the spherical summation of Fourier series in two dimensions fails for some $f\in L^p(\bbT^2)$ if $p\neq 2$. Naturally one is then led to consider regularizations of the disc multiplier $\bbone_D$.

We consider two such regularizations 
and ask for their precise behavior in $M_p$ for $p\neq 2$. 
First, fix  $\eta\in L^\infty(\bbR)$   with support in $[-1/2,1/2]$ and  $\int \eta=1$,  and   let  $\delta\ll 1/2$
(see Section~\ref{sec:appl} for a further relaxation of the assumptions on $\eta$). We then set $\eta_\delta(u)=\delta^{-1}\eta(\delta^{-1}u)$  and $\omega_\delta =\bbone_{[-1,1]}*\eta_\delta$ and consider the radial extension of $\omega_\delta$ as a multiplier in $\bbR^2$, thus producing 
 the {\it mollified  disc multipliers}  
\begin{equation}
    h_\delta(\xi)=\om_\delta(|\xi|).
\end{equation} 
Note that $h_\delta(\xi)=1$ for $|\xi|\le 1-\delta$ 
and $h_\delta(\xi)=0$  for $|\xi|\ge 1+\delta$, moreover $h_\delta$  is Lipschitz continuous with Lipschitz constant $\lc \delta^{-1} \|\eta\|_\infty$.  One may want to choose for $\eta$ a smooth function with compact support, but there are also interesting rougher choices such as $\eta=\bbone_{[-\frac 12,\frac 12]}$,  in which case we get a trapezoidal (de la Vall\'ee-Poussin type)  profile with 
\[ h_\delta(\xi)= \begin{cases}
1 &\text{ if } |\xi| <1-\frac{\delta}2,
\\- \frac{|\xi|}{\delta} +\frac{2+\delta}{2\delta}&\text{ if } 1-\frac \delta 2\le |\xi|\le 1+\frac \delta 2,
\\0  &\text{ if }|\xi|\ge 1+\frac \delta 2.
\end{cases} 
\]

Secondly,  fix $\alpha\in (0,\frac 12)$ and consider the well-studied 
{\it Bochner--Riesz multiplier} at index $\alpha$,  
\begin{equation}
    m_\alpha(\xi)= (1-|\xi|^2)_+^\alpha.
\end{equation}

Sharp results on both multipliers are known for  $1\le p<4/3$ and its dual range $p>4$. We have   $\|h_\delta\|_{M_p(\bbR^2)} \approx_p \delta^{-\alpha_p}  $ with $\alpha_p=2|\frac 1p-\frac 12|-\frac 12$ (at least for smooth $\eta$) and  
$m_\alpha\in M_p$ for $\alpha>\alpha_p$ 
(see \cite{fefferman69}, \cite{CarlesonSjolin}, \cite{FeffermanBR73}, \cite{CordobaBR79}). 
Regarding the $\alpha=\alpha_p$ endpoint  for $1\le p<4/3$  sharp weak type $(p,p)$ estimates for $m_{\alpha_p}$,  and their strengthening via essentially optimal  sparse domination results  are known, and also optimal  $M_p$ results  for multipliers such as $(1-|\xi|^2)^{\alpha_p}_+ |\log \tfrac{1}{|1-|\xi|^2|}|^\beta$, with 
$\beta>\max\{\frac 1p, 1-\frac 1p\}$. 
See \cite{Christ-rough, Christ-BR87}, \cite{Seeger-Indiana, seeger-BRwt}, \cite{Tao-Indiana1998},  \cite{BeltranRoosSeeger-sparseBR, BRS-OWF}.
For   historical context see also  \cite[\S9.3]{Carbery-LMSsurvey}.

Much less is known  about  the   behavior  of the $M_p$-norms of $h_\delta$ and $m_\alpha$ in the range $4/3\le p\le 4$ as  $\delta\to 0+$, $\alpha\to 0+$.    
Some preliminary results can be obtained from a result of C\'ordoba \cite{CordobaBR79} which we now describe.  
Let $\chi_0\in C^\infty_c (\bbR)$ be a nonnegative bump function supported in $(-1,1)$. 
Then  
\cite{CordobaBR79} provides  
the upper bound in 

\begin{equation} \label{eq:cord} \big\|\chi_0 (2^j (1-|\cdot |^2))\big \|_{M_p(\bbR^2)} \approx j^{|\frac 1p-\frac 12| }. 
\end{equation}
The lower bound in \eqref{eq:cord} also holds,  by  
a variant of Fefferman's  argument for the disc multiplier \cite{Fefferman-ball} combined  with  sharp Minkowski  dimension type lower bounds for the Besicovitch set by  Keich \cite{Keich99}. A detailed  proof of this lower bound, for nonnegative nonzero $\chi_0$ can be found for example in  
\cite[Lemma 2.3]{BeltranCarberyRoncalSeeger}, see also \cite{LaSalle, FernandezLaSalle}.

Assuming in addition $\supp(\chi)\subset (\tfrac 14,1) $ and summing estimates \eqref{eq:cord}, with suitable choices of $\chi$,  yields for $4/3\le p\le 4$ the upper bound in 
\begin{equation} \label{eq:cordsum}
(\log\tfrac 1\delta)^{|\frac 1p-\frac 12|} \lc \|h_\delta \|_{M_p} \lc  (\log \tfrac 1\delta) ^{5|\frac 1p-\frac 12|}.
\end{equation}  
Likewise we get $\|m_\alpha \|_{M_p} \lc \alpha^{-5|\frac 1p-\frac 12|}.$ For the  lower bound in \eqref{eq:cordsum} one observes
that $ u_\delta(\xi):=h_\delta(\xi) -h_\delta((1+\delta)\xi) $ is a multiplier of the form  used   in  \eqref{eq:cord}, and by the dilation  invariance of the multiplier norms, $\|u_\delta\|_{M_p}\le 2 \|h_\delta\|_{M_p}$. 
For the cases $p=4/3$ and $p=4$, an asymptotic analysis yields that the convolution kernel $\cF^{-1}[h_\delta]$ has $L^{4/3}$ norm $\approx (\log\tfrac 1\delta)^{3/4}$. Testing the operator on suitable Schwartz functions then  leads to 
$\|m\|_{M_4}=\|m\|_{M_{4/3}}\gc (\log\tfrac 1\delta)^{3/4}$. Our main result gives further improvements and completely closes the gap in \eqref{eq:cordsum} for the limiting cases $p=4/3$ and $p=4$.

\begin{thm} \label{thm:main} 
 There is a constant $C$ such that  for $0<\delta\le \frac 12$, 
\begin{equation*}    C^{-1}\log \tfrac 1\delta \le \|h_\delta \|_{M_{4}(\bbR^2)} \le C \log\tfrac 1\delta. \end{equation*} 
\end{thm}

The upper bounds in Theorem  \ref{thm:main}  follow from a more general result on sums of multipliers featured in \eqref{eq:cord}, where we may  also impose less restrictive regularity assumptions on the $\chi_j$ to obtain the rougher versions of  Theorem \ref{thm:main}.
In what follows we denote by $L^4_a(\bbR)$ the standard $L^4$-based Sobolev (i.e.  Bessel-potential) space on the real line. 

\begin{thm}  \label{thm:Tjthm} Let $a>1/2$  and let $\cU$ be a bounded set of $L^4_a(\bbR)$ functions supported in $(\frac 18, \frac 12)$.
Define the convolution operator $T_j$ acting on $\cS(\bbR^2)$ by 
\[ \widehat{T_jf} (\xi)= \om_j(2^j(1-|\xi|) ) \widehat f(\xi)\] 
with $\om_j\in \cU$. Then there is a constant $C$ (depending only on $\cU$) such that for $n\ge 1$, $N\ge 1$, 
\begin{equation}\label{eq:impr} 
    \Big\|\sum_{n< j\le n+N} T_j \Big \|_{L^4\to L^4}  \le C (n+N)^{1/4} N^{3/4}. 
\end{equation}
\end{thm}

For the application of Theorem \ref{thm:Tjthm} to Theorem \ref{thm:main}  we will need to use  the continuous embedding 
$\mathrm{Lip(1)}\hookrightarrow L^4_{\eps+1/2} $  for $\eps<1/4$. When  $N=1$ in Theorem \ref{thm:Tjthm} we  recover  \eqref{eq:cord},  but for larger $N$  we get an improvement over just   using the triangle inequality.  Note that the Fourier transform $\widehat{T_j f}$ is supported in the intersection of a sector with the  thin annulus 
$A_j=\{\xi: 1- 2^{-j-3}\le |\xi|\le 1-2^{-j-1}\}.$
A key idea is to somehow exploit the essential orthogonality of the operators $T_j$ (or related versions of it) in $L^2$, but  it is  not {\it a priori} obvious how to use it for the $L^p$ inequalities.

\begin{remarka} Note an  easy  corresponding version  for multipliers on the real line, namely $\sum_{j>2} \om_j(2^j(1-|\xi|) )$ belongs to $M_p(\bbR)$, for $1<p<\infty$.  
We  consider $m_\pm(\xi) =\sum_{j>2} \omega_j (2^j(1\pm \xi))$ and observe that the multipliers $m_\pm(\xi\mp 1)$ satisfy  the assumptions of the   Marcinkiewicz multiplier theorem.
\end{remarka}

Theorem \ref{thm:Tjthm} implies  the upper bounds in Theorem \ref{thm:main} and also  better  bounds for $\|m_\alpha\|_{M_4}$. After an  interpolation with $L^2$ estimates we get the following improvement over \eqref{eq:cordsum}. The argument for it is sketched in Section~\ref{sec:upperbounds}. 
\begin{cor} \label{cor:interpol}   Let $0<\delta, \alpha<\frac 12$ and    $\tfrac 43\le p\le 4$. Then the following hold.

(i) 
\[ \|h_\delta\|_{M_p(\bbR^2)} \lc (\log \tfrac 1\delta) ^{4|\frac 1p-\frac 12|} ,\quad 
\|m_\alpha \|_{M_p(\bbR^2)} \lc \alpha^{-4|\frac 1p-\frac 12|}.\]

(ii) Define the logarithmic Bochner--Riesz multipliers  by \[ \ga_\beta(\xi)= 
\begin{cases} 
|\log(1-|\xi|^2)|^{-\beta} &\text{ for  $|\xi|< 1$}\\ 
0 &\text{ for  $|\xi|\ge 1$.}
\end{cases} 
\]Then $\ga_\beta\in M_p$ for $\beta>4|\frac 1p-\frac 12|$. 
\end{cor}

\begin{remarka}
   The examples showing the lower bounds for the dual $L^{4/3}$ estimate in  Theorem \ref{thm:main}   are radial.  Thus  if $\cT_\delta$ denotes the convolution operator with multiplier $h_\delta$ then
    \Be\label{eq:equiv}\|\cT_\delta\|_{L^{4/3} \to L^{4/3} } \approx \|\cT_{\delta} \|_{L^{4/3}_\rad \to L^{4/3} }\Ee and a similar statement for $p=4$. 
    This mirrors the  equivalence  of the $L^p\to L^p$ and the $L^p_\rad\to L^p_\rad$ operator norms in the previously mentioned 
     results for the range  $1\le p<4/3$. However, in  that range  both operator norms are equivalent with $\|K_\delta\|_p$,  the $L^p$-norm of the  convolution kernel.  The latter equivalence breaks down for the endpoint $p=4/3$, indeed, as noted above,   $\|K_\delta\|_{4/3}\approx (\log \tfrac 1\delta)^{3/4}$ but $\|\cT_\delta\|_{L^p\to L^p} \approx \log \tfrac 1\delta$. The analogue of the equivalence \eqref{eq:equiv} for $4/3<p<4$ also breaks down, as does the analogous statement for the  restricted weak-type category for $p=4/3$ or  $p=4$. This follows by the positive results of 
        Herz \cite{Herz-ball} and Chanillo \cite{Chanillo-ball} for radial functions and the negative results for general functions based on  the Kakeya-type examples (\cite{Fefferman-ball}, \cite{BeltranCarberyRoncalSeeger}).     
\end{remarka}
\subsubsection*{Notation} Above and throughout this paper we write, for  nonnegative  quantities $a,b$,  $a \lesssim b$ or $a\lesssim_L b$
  to indicate $a \leq C b$ for some constant $C$ which may depend  on some list $L$. We write  
  $a \approx  b$ to indicate that both $a \lesssim b$ and $b \lesssim a $ hold. 

\subsubsection*{Declaration of research tools} No AI tools were used in the preparation of this paper. It was  typeset  using  \AmS-\LaTeX.

\section{Lower bounds for mollified ball multipliers} \label{sec:lower-bounds}
Here we prove the lower bounds in Theorem \ref{thm:main}. The arguments here work equally well in higher dimensions $d$ for $\delta$-regularizations  of the ball multiplier and give the  lower bound 
\begin{equation} \|h_\delta\|_{M_{{p_d} }}\ge c \log \tfrac 1\delta, \quad p_d= \tfrac{2d}{d-1} \end{equation}   for some positive constant $c$ and sufficiently small $\delta>0$. 
We will derive such lower  bounds 
by  approximating the ball multiplier with its mollifications.  Specifically we rely on  a paper by Kenig and Tomas \cite{KenigTomas-ball} who proved the failure of the weak type $(p_d,p_d)$ property  for the ball multiplier operator acting on radial functions.

We review a few formulas for radial multiplier transformations acting on radial functions and apply them to the ball multiplier and its mollification. 
 Recall that for a radial function $f(x)= g(|x|)$ the Fourier  transform  is given by 
$\widehat f(\xi)= (2\pi)^{d/2} \cB_g g(|\xi|)$ where
\[\cB_d g(\rho)= \int g(s) B_d(s\rho)\rho^{d-1} d\rho, \, \text{ with } B_d(\varrho)=  \varrho^{-\frac{d-2}{2}} J_{\frac{d-2}{2}}(\varrho).\]
Here we adopt convenient notation in \cite{GarrigosSeeger2008},  noting that $B_d$ corresponds up to a constant to the radial profile of the Fourier transform of surface measure on the unit sphere. 
The Fourier inverse is then given by $\F^{-1} f(\xi)=(2\pi)^{-d/2} \cB_d g(|\xi|)$. Hence, for a radial multiplier $\om(|\cdot|)$ we then have 
\begin{subequations}\label{eq:Fourier-BesseL}
\begin{equation}
    \label{eq:Fourier-Bessel-expr} 
\cF^{-1}[\om(|\cdot|) \widehat f] (x)= \int \cK_\om(r,s) s^{d-1} g(s) ds, \quad \text{ with }r=|x|,
\end{equation} where 
\begin{equation} \label{eq:Fourier-Bessel-kernel} \cK_\om(r,s)= \int_0^\infty \om(\rho) B_d(r\rho) B_d(s\rho) \rho^{d-1} d\rho. 
\end{equation} 
\end{subequations} 
For the ball multiplier we use this for $\om=\bbone_{[0,1]}$ or, equivalently for $\om_0(\rho)= \bbone_{[-1,1]}$. 
Modifying  a definition in  \cite{KenigTomas-ball} 
we  set 
\begin{equation}  g_\delta(s)= \begin{cases} s^{-\frac{d+1}{2}} \cos (s- \tfrac{d-1}{4}\pi), &\text{ if } \delta^{-1/4} \le s\le \delta^{-1/2} ,\
\\
0&\text{ otherwise}. 
\end{cases} 
\end{equation} 
and $ f_\delta(y)=g_\delta(|y|)$. 
Note that
\Be\label{eq:fdelta} \|f_\delta\|_{p_d}\lc (\log \tfrac 1\delta)^{\frac{1}{p_d}}.
\Ee

We will rely on a lower bound for the ball  multiplier operator acting on the $f_\delta$ when evaluated on 
\begin{equation} \label{def:Adelta} \fA(\delta)= \{x: 2\delta^{-1/2}\le |x|\le \delta^{-1} \}.
\end{equation} 

\begin{lem} \label{lem:lower-disc} 
Assume $d\ge 2$, $p_d=\frac{2d}{d+1}$. 
Then there exists $c>0$  and $\delta_\circ>0$ such that for all  $0<\delta\le \delta_\circ $
\Be\label{lowerbddisc}\Big(
\int_{\fA(\delta)} \big|\cF^{-1} [\bbone_B \widehat f_\delta](x)\big|^{p_d} dx\Big)^{1/p_d}\ge c (\log \tfrac 1\delta)^{1+\frac 1{p_d}}.
\Ee
\end{lem}
\begin{proof} 
We have to consider $\cK_{\om_0} (r,s) $ for $r\ge 2s$. This is (essentially) in \cite{KenigTomas-ball}.
From  \cite[\S3.21 (5),(6)]{Watson-treatise} we have 
\[ \frac{d}{dz} \big( z^d B_{d+2}(z)) = B_d(z) z^{d-1}, \qquad  \frac{d}{dz} [B_d(z)]=- zB_{d+2}(z)\]
and, following \cite{KenigTomas-ball} an integration by parts yields
\begin{align*}
\cK_{\om_0}(r,s)&=s^{-d} \int_0^1 B_d(r\rho) \frac{d}{d\rho}\big[ (s\rho)^d B_{d+2} (s\rho)  \big]  d\rho\\&=
s^{-d} B_d(r) s^d B_{d+2}(s)-  s^{-d} \int_0^1 \frac{d}{d\rho}\big[ B_d(r\rho)] (s\rho)^d B_{d+2}(s\rho) d\rho\\&= 
B_d(r) B_{d+2} (s) + r^2 a(r,s)
    \end{align*}
    where $a(r,s)=\int_0^1 B_{d+2}(r\rho) B_{d+2}(s\rho) \rho^{d+1} d\rho$.  
    This yields
    \[\frac{ \cK_{\om_0} (r,s)- B_d(r) B_{d+2}(s)} {r^2}= a(r,s)=a(s,r)= \frac{\cK_{\om_0}(s,r) - B_d(s) B_{d+2}(r)}{s^2} \]
and since $\cK_{\om}(r,s)=\cK_{\om} (s,r)$  for any $\om$ we obtain
\Be \label{eq:Kexplicit}
\cK_{\om_0}(r,s)= 
\frac{r^2}{r^2-s^2} B_{d+2} (r) B_d(s) -\frac{s^2}{r^2-s^2} B_d(r) B_{d+2}(s).
\Ee

Hence, for $x\in \fA(\delta)$,  \[ \cF^{-1} [\bbone_B\widehat {f_\delta} ](x) = \cT_0 g_\delta(x)- \cT_1 g_\delta(x) \]
where 
\begin{align*} \cT_0 g_\delta(x)&= B_{d+2}(|x|) \int_0^\infty\frac{|x| ^2}{|x|^2-s^2}   B_d(s) g_\delta(s) s^{d-1} ds, 
\\
\cT_1 g_\delta(x) &= B_d(|x|) \int_0^\infty \frac{s^2}{|x|^2-s^2} B_{d+2}(s) g_\delta(s)  s^{d-1}  ds.
\end{align*} 
Observe that  no singularities occur in these integrals since we assume $|x|\ge 2\delta^{-1/2}$ and $g_\delta$ is supported in $[\delta^{-1/4}, \delta^{-1/2}].$

Using asymptotics for Bessel functions \cite[ch. IV, Lemma 3.11]{stein-weiss} we have for $s\ge 1$
\[B_d(s) = \sqrt{\frac 2 \pi}  \frac{ \cos(s-\frac{d-1}{4} \pi)}{s^{(d-1)/2} }+ O( s^{-(d+1)/2}).\]
From the definition of $g_\delta$, $|\cT_0 g_\delta(x)|$ is bounded below by \[\left(\frac2\pi\right)^{1/2} |B_{d+2}(|x|)| \bigg( \int_{\delta^{-1/4}}^{\delta^{-1/2}} \frac{|x|^2}{|x|^2-s^2} 
\cos^2(x-\tfrac{d-1}{4}\pi) \frac{ds}{s} - C \int_{\delta^{-1/4}}^{\delta^{-1/2}} s^{-2} ds\bigg) \]
 and thus  $|\cT_0 g_\delta(x)|\gc |B_{d+2}(x)|\log \tfrac 1\delta$ for $x\in \fA(\delta)$. By the asymptotics for $B_{d+2}(r)$ we have
\[B_{d+2} (r) = \sqrt{\frac 2 \pi}  \frac{ \cos(r-\frac{d+1}{4} \pi)}{s^{(d+1)/2} }+ O( s^{-(d+3)/2})\] 
which  leads to 
\Be \label{eq:T0low} \Big(\int_{\fA(\delta)} | \cT_0 g_\delta  (x)|^{p_d} dx\Big)^{\frac1{p_d}} \ge c (\log \tfrac 1\delta)^{1+\frac 1{p_d} } .
\Ee 

For $\cT_1 g_\delta(x)$ we derive  instead an upper bound. Using $B_d(r)=O(r^{-\frac{d-1}{2}})$, $B_{d+2}(s) =O(|s|^{-\frac{d+1}{2}})$ we get 
\[ |\cT_1 g_\delta(x)| \lc 
\delta^{-1/2} |x|^{-\frac{d+3}2} 
\]  and a straightforward calculation yields 
\Be \label{eq:T1upp}\Big(\int_{\fA_\delta} | \cT_1g_\delta (x)|^{p_d} dx\Big)^{\frac1{p_d}} \le C, \Ee uniformly in $\delta$. Thus  for small $\delta$, 
\eqref{lowerbddisc} follows from \eqref{eq:T0low} and \eqref{eq:T1upp}. 
\end{proof}

\subsection*{Lower bounds for the mollified ball multipliers}
For the lower bound $\approx \log\tfrac 1\delta$ we may assume, by continuity considerations,  that $\delta>0$ is small. We do not need a strong assumption on $\eta$ here and just  assume that $\eta\in L^1$ supported on $[-\frac 12,\frac 12]$. 
By \eqref{eq:fdelta} the lower bound follows from
$\|\F^{-1}[h_\delta f] \|_{p_d}\gc (\log \tfrac 1\delta)^{1+1/p_d} $ and by Lemma \ref{lem:lower-disc} and the triangle inequality this is a consequence of 
\begin{lem}\label{lem:upperbound-approximation}
Assume $d\ge 2$, $p_d=\frac{2d}{d+1}$. 
Then there exists $c>0$  and $\delta_\circ>0$ such that for all  $0<\delta\le 1/2 $
\Be\label{upperbddiffdisc}
\Big(\int_{\fA(\delta)} \big|\cF^{-1} [(h_\delta-\bbone_B) \widehat f_\delta](x)\big|^{p_d} dx\Big)^{1/p_d}\lc \log \tfrac 1\delta
\Ee
    \end{lem} 
\begin{proof} 
We use \eqref{eq:Fourier-BesseL} for $h_\delta=\om_\delta(|\cdot|)$ with \[\om_\delta (\rho)=\int \bbone_{[-1,1]}(\rho-u) \eta_\delta (u) du.\]  Note that 
\[ \bbone_{[-1,1]} (\rho-u)= \bbone_{[0, 1+u]}(\rho) \text{ if } \rho>0, |u|\le 1\] and we compute
\begin{align*}
    \cK_{\om_\delta}(r,s)&= \int \eta_\delta(u) \int_0^{1+u}  B_d(r\rho) B_d(s\rho) \rho^{d-1} \,d\rho\, du
    \\&=
    \int \eta_\delta(u) \int_0^{1}  B_d(r(1+u)\varrho) B_d(s(1+u)\varrho) (1+u)^d \varrho^{d-1} \,d\varrho\, du, 
\end{align*}
and thus
\Be\label{scaleddisc} \cK_{\om_\delta}(r,s) = \int \eta_\delta(u) (1+u)^d \cK_{\om_0}(r(1+u), s(1+u) ) \,du .\Ee

Since $\int \eta_\delta(u) du=1$ we can write 
\begin{align*}
    &\F^{-1} [(h_\delta-\bbone_B)\widehat {f_\delta} ](x)
    \\&= \int \eta_\delta(u) \int \Big( \cK_{\om_\delta }(|x|,s) -\cK_{\om_0} (|x|,s) \Big) g_\delta(s) s^{d-1} ds\, du 
    \\&= \int u \eta_\delta(u)\int \frac{d}{dv} \Big[ (1+v)^d \cK_{\om_0}(|x|(1+v), s(1+v))\Big]_{v=1+\sigma u} d\sigma  \,g_\delta(s) s^{d-1} ds\, du 
\end{align*}
where we have used \eqref{scaleddisc} and the fundamental theorem of calculus. 
Using $\partial_1$, $\partial_2$ as notation for derivatives with respect to the first and second entry of $\cK_{h_0}(r,s)$ we have 
\begin{multline*}
   E(v, r, s):=  \frac{d}{dv} \Big[ (1+v)^d \cK_{h_0} (r(1+v), s(1+v))\Big]
    \\=
    d(1+v)^{d-1} \cK_{h_0} + r\partial_1 \cK_{h_0} + s\partial_2\cK_{h_0} \Big|_{(r(1+v), s(1+v))} .
\end{multline*}
We use \eqref{eq:Kexplicit} to compute these expressions and then use \[ |B_{d+2}(r)|+|B_{d+2}'(r) |\lc r^{-\frac{d+1}{2}},\quad 
|B_{d}(r)|+|B_{d}'(r) |\lc r^{-\frac{d-1}{2}},\]  and similar estimates for the same expressions evaluated at $s$. The outcome is that
\[ |E(v,r,s)|\lc r^{-\frac{d-1}{2}} s^{-\frac{d-1}{2} }  \text{ provided that  $2s\le r$, $|v|\le \delta$.}
 \]
From  $|g_\delta(s) s^{d-1} | \le s^{\frac{d-3}{2}}$ we get  
\begin{multline*}\big|\F^{-1} [(h_\delta-\bbone_B) \widehat {f_\delta} ](x)\big|
\\\lc \int_{|u|\le \delta} |u||\eta_\delta(u)| 
|x|^{-\frac{d-1}{2} }\int_{\delta^{-1/4}}^{\delta^{-1/2} }\frac{ds}{s} du
\lc |x|^{-\frac{d-1}2}  \delta \log \tfrac 1\delta .\end{multline*}
Since $(\int_{\fA(\delta)} |x|^{- \frac{d-1} {2} p_d} dx)^{1/p_d}= O(\delta^{-1})$ we obtain \eqref{upperbddiffdisc}.
\end{proof}

\section{Upper bounds}\label{sec:upperbounds} 
\subsection{The applications of Theorem \ref{thm:Tjthm}} \label{sec:appl}  We  show how Theorem \ref{thm:Tjthm} implies the asserted upper bounds in $M_4$  for the mollified disc multiplier and the two kinds of Bochner--Riesz multipliers. 

We first consider $h_\delta(\xi)= \om_\delta(|\xi|)$ with $\omega_\delta= \bbone_{[-1,1]}*\delta^{-1}\eta(\delta^{-1}\cdot)$, $\supp(\eta)\subset [-\frac 12, \frac 12]$ and $\delta<1/4$. 
We show the upper bound $\|h_\delta\|_{M_4} \lc \log\delta^{-1} $ under the less restrictive assumption  $\eta\in L^r$ for $r>4/3$.
 For the application of Theorem \ref{thm:Tjthm} it is convenient to scale and consider the multiplier 
$\widetilde h_\delta(\xi) = \om_\delta( (1+2\delta) |\xi|) $ which is supported in $\{\xi: |\xi|\le \frac{1+\delta}{1+2\delta}\}$ and has the same $M_4$ norm. 
We can write 
\Be\label{eq:hdeltadec}\cF^{-1} [\widetilde h_\delta \widehat f ] = P_0 f+\sum_{j\ge 1} T_j f \Ee
where $P_0$ is the operator of convolution with a Schwartz function
and
\[ \widehat {T_j f} = \chi_1(2^j(1-|\xi|) )\widetilde h_\delta(\xi) \widehat f(\xi)\]
with $\chi_1\in C^\infty_c$, $\supp(\chi_1)\in [\tfrac 18, \tfrac 12]$.
Now 
\Be\label{eq:th-dec}\notag\begin{aligned}  \chi_1(2^j (1-|\xi|) \widetilde h_\delta(\xi) 
&= \chi_1(2^j (1-|\xi|) ) \int_{-1}^{1} \delta^{-1} \eta(\delta^{-1}
((1+2\delta)|\xi|-u)) du\\ 
&=\om_j(2^j(1-|\xi|)) \end{aligned}\Ee 
where (with the change of variable $\sigma=2^j(1-\rho))$
\Be\label{eq:th-dec2}\om_j(\sigma)=  \chi_1(\sigma)\int_{-1}^1 \delta^{-1} \eta\big(\delta^{-1} (1+2\delta)(1-2^{-j}\sigma)-\delta^{-1}u\big)du. 
\Ee
Note that the integral in \eqref{eq:th-dec2} is equal to $0$ for 
$\frac{1+\delta}{1+2\delta} \le 1-2^{-j-1}$, i.e. for $2^j>\delta^{-1}$.  The integral is  equal to $1$ when
$1-2^{-j}\sigma<\frac{1-\delta}{1+2\delta}$. Thus $\om_j(2^j(1-|\xi|)= \chi_1(2^j(1-|\xi|)$ if $1- 2^{-j-3}\le \frac{1-\delta}{1+2\delta} $
which holds when  $ 2^j\le (24\delta)^{-1}$. It remains to determine the regularity of $\om_j$ for $(24\delta)^{-1}\le 2^j \le \delta^{-1}$. Note that the integral in \eqref{eq:th-dec2} is bounded and we get for the weak derivative 
\[ \frac{d}{d\sigma} \omega_j (\sigma)= \chi_j(\sigma) \eta ( (1+2\delta) \tfrac{1-2^{-j}\sigma}{\delta} -\tfrac 1\delta) \tfrac{1+2\delta}{2^j\delta} + \chi_j'(\sigma) e_j(\sigma) \]
 where $e_j$ is uniformly bounded. Since now $2^j\delta\approx  1$ we have  $\|\om_j\|_{L^r_1}\lc \|\eta\|_r$. By the Sobolev embedding theorem, $L^r_1\hookrightarrow L^4_a$ for $a=1-\frac 1r+\frac 14$ and since $r>4/3$ we see that $a>1/2$. 
Consequently  we may apply  Theorem \ref{thm:Tjthm} to the sum $\sum_{j=1}^N T_j$ with $n=1$, $N=1+\floor{ \log\tfrac 2\delta}$ and obtain  the upper bounds in Theorem \ref{thm:main}

For the Bochner--Riesz multipliers we make different choices of $\chi_j \in \cU$ and write \[\cF^{-1}[m_\alpha \widehat f] = \widetilde P_0 f+\sum_{k=0}^\infty 2^{-k} \sum_{\frac k\alpha < j \le \frac{(k+1)}\alpha} T_j f\] where we now apply 
Theorem \ref{thm:Tjthm} with  
$N\sim \floor{\alpha^{-1}}$, $n=\floor{ k\alpha^{-1}}$  and then sum in $k$, obtaining $ \|m_\alpha\|_{M_4} \lesssim \sum_{k=0}^\infty 2^{-k} (\frac{k+1}{\alpha})^{1/4}(\frac{1}{\alpha})^{3/4} \sim \alpha^{-1}$. All other estimates in part (i) of Corollary \ref{cor:interpol} follow by interpolation with $L^2$ estimates and duality. 

 Finally, for the logarithmic  Bochner--Riesz operators in part (ii) of Corollary \ref{cor:interpol}  we decompose
\[\cF^{-1}[\ga_{\beta} \widehat f] = \widetilde P_0 f+\sum_{k=1}^\infty 2^{-k\beta } \sum_{ 2^k+1\le  j \le 2^{k+1} } T_j f\] 
for suitable choices of $\eta_j$ in a bounded subset of $C^\infty_c$ supported in $(\frac 18, 1)$. 
We then use Theorem \ref{thm:Tjthm} for the $k$-blocks $\sum_{ 2^k+1\le  j \le 2^{k+1} } T_j $,  with $n=2^k=N$,  and obtain the $L^4\to L^4$ operator norm  $O(2^k)$.
By interpolation the $L^p\to L^p$ operator  norm of these $k$-blocks is $O(2^{4k|1/p-1/2|})$ and since $\beta>4|1/p-1/2|$
we can sum in $k$.

\subsection{Proof of Theorem \ref{thm:Tjthm}} We use a variant of arguments in \cite{seeger-BRwt} (and \cite{CarberySeeger-QJ} on a related problem on weighted norm inequalities for Bochner--Riesz square functions). Observe that  $\widehat{T_j f}$ is supported in  the  annulus 
\[A_j=\{\xi: 1- 2^{-j-3}\le |\xi|\le 1-2^{-j-1}\}.\]
The key is to use the essential disjointness of the annuli $A_j$ so that some  weak orthogonality property can be exploited. The orthogonality is used to prove an inequality for vector-valued operators:

\begin{prop} \label{thm:vTjthm} Let $\cU$, $T_j$ be as in Theorem \ref{thm:Tjthm}. Then for $f\in L^4(\bbR^2)$, $N\ge 2$,
\begin{equation}\label{eq:afterH}
    \Big(\sum_{j=1}^N \|T_j f\|_4^4 \Big)^{1/4} \lc N^{1/4} \|f\|_4.
\end{equation}
\end{prop}
Theorem \ref{thm:Tjthm} follows by crudely applying H\"older's inequality for the $j$-sum and interchanging   sum. This gives 
\begin{align}\notag
    \Big\|\sum_{n< j\le n+N} T_j f\Big \|_{4}  &\lc  N^{3/4}  
    \Big(\sum_{n<j\le n+N} \|T_j f\|_4^4 \Big)^{1/4}
    \\ \label{eq:applv} &\lc N^{3/4} (N+n)^{1/4} \|f\|_4
\end{align}
where in \eqref{eq:applv} we have applied Proposition \ref{thm:vTjthm} with $N$ replaced with  $N+n$. 

\subsubsection{\it Proof of Proposition \ref{thm:vTjthm} }
Let $\chi_\circ$ be a $C^\infty$ function supported in a ball of radius $\frac 1{10}$ centered at $(0,1)$. 
By a localization and rotation argument we may replace the operators $T_j$ in the formulation of the theorem with $\cT_j$ defined by
\begin{equation}\label{eq:bilinearfirst}
\cT_j f (x)= \cB_j [\om_j,f](x)  
\text{ where }  
\widehat{\cB_j[\om,f] } (\xi) =
\chi_\circ (\xi)
\om(2^j(1-|\xi|) ) \widehat f(\xi),    
\end{equation}  
with $\|\om_j\|_{L_{\eps+1/2}^4} \lc 1$. Below it will be useful 
 to view $\cT_j$ as an bilinear operator acting on $\omega_j$ and $f$ as in \eqref{eq:bilinearfirst}.

Observe that  $\widehat{\cT_j f}$ is supported in the intersection of a narrow sector with the  annulus 
$A_j$.

Let $\vartheta\in C^\infty_c(\bbR)$ satisfy  $\sum_{\nu} \vartheta(s-\nu)=1$ and $\supp(\vartheta)\subset (-1,1)$ 
and define \[ \widehat{T_{j,\nu} f} = \vartheta(2^{j/2} \xi_1 -\nu) \widehat {\cT_j f}. \] 
By the Fefferman argument in \cite{FeffermanBR73} (using  the essential disjointness of the sets $\supp \,\widehat {T_{j,\nu} f} +\supp \, \widehat {T_{j,\nu'} f} $ due to the curvature of the circle),
\begin{equation} \label{eq:Feff}
\|\cT_j f\|_4^4= \Big\|\sum_{\nu,\nu'} T_{j,\nu} f\, T_{j,\nu'} f \Big\|_2^2 \lc \sum_{\nu,\nu'} \big\|T_{j,\nu} f\, T_{j,\nu '}f\big\|_2^2.
\end{equation} 
Let $\Gamma_j^0=\{ (\nu,\nu)\}$ and $\Gamma_j^m = \{(\nu,\nu'): 2^{m-1} \le |\nu-\nu'| < 2^{m}\}$, where in this definition  $1\le 2^{m-1} \le 2^{\frac j2+C} $. 
If $R_{j,\nu}$ is the support of the multiplier for $T_{j,\nu}$ the distance of $R_{j,\nu}$ and $R_{j,\nu'}$ is $\approx 2^{m-\frac j2}$ provided that $m>C$. 

For $(\nu,\nu')\in \Gamma_j^m$ we decompose
\[ T_{j,\nu} f= \sum_\mu T^m_{j,\nu,\mu} f\] where 
\[ \widehat{ T^m_{j,\nu,\mu}f}(\xi) = \vartheta(2^{\frac j2+m} \xi_1-\mu) \widehat {T_{j,\nu} f} (\xi).\]
We  argue  as in \cite{seeger-BRwt} and note that for fixed $m$ the sets
\[\supp \,\widehat {T_{j,\nu,\mu} f} +\supp\, \widehat {T_{j,\nu',\mu'} f} \quad \text{  with $(\nu,\nu') \in \Gamma^m_j$ } \] are essentially disjoint. Therefore, again by the Fefferman argument, 
\[ 
\sum_{\nu,\nu'} \big\|T_{j,\nu} f\, T_{j,\nu '}f\big \|_2^2\lc  \sum_{0\le m\le \frac j2+C} \sum_{(\nu,\nu') \in \Gamma_j^m } \sum_{\mu,\mu'}
\big \|T^m_{j,\nu,\mu} f\, T^m_{j,\nu ',\mu'}f\big \|_2^2.
\]

Now define, for  $\kappa\in [-C2^{\frac j2-m}, C2^{\frac j2-m}]\cap\bbZ $, 
\[\fV^\ka_{j,m}  =\big\{\mu\in \bbZ: |2^{-\frac j2-m} \mu- 2^{m-\frac j2}\ka|\le 2^{m-\frac j2+1} \big\} .\]
Note that for fixed $j,m$ every $\mu$ is contained in $O(1) $ of the sets $\fV^\ka_{j,m}$ (and contained in at least one of those). 
Moreover, for fixed $j,m,\mu$,  $T^j_{\nu,\mu}$ is nonzero only for $O(1)$ of the $\nu$. 

We can thus replace the previous bound by 
\begin{align*} 
    \sum_{\nu,\nu'} &\big\|T_{j,\nu} f\, T_{j,\nu'}f\big\|_2^2   \\  &\lc  \sum_{i\le C} \sum_{0\le m\le \frac j2+C} \sum_{-C2^{m-\frac j2}\le \ka \le C 2^{m-\frac j2}}\\ 
 &\qquad\qquad \Big\| \Big(\sum_{\mu\in \fV^\ka_{j,m} }\sum_\nu \big| T^m_{j,\nu,\mu} f \big|^2\Big)^{1/2} \Big(\sum_{\mu'\in \fV^{\ka+i} _{j,m} } \sum_{\nu'} \big| T^m_{j,\nu',\mu'} f |^2\Big)^{1/2} \Big\|_2^2
\\&\lc 
    \sum_{0\le m\le \frac j2+C} \sum_{-C'2^{m-\frac j2}\le \ka\le  C' 2^{m-\frac j2}}
    \Big\| \Big(\sum_{\mu\in \fV^\ka_{j,m} } \sum_\nu \big| T^m_{j,\nu,\mu} f\big |^2\Big)^{1/2} \Big\|_4^4 \,.
\end{align*}
Putting this into \eqref{eq:Feff} 
we get 
\begin{multline*}
   \Big(\sum_{j\le N} \|\cT_j f\|_4^4 \Big)^{1/4}\,\lc
\\ \bigg(\sum_{0\le m\le C+\frac N2} 
   \sum_{2m-C\le j\le N }  \sum_{-C'2^{m-\frac j2}\le \ka\le  C' 2^{m-\frac j2}}
     \Big\| \Big(\sum_{\mu\in \fV^\ka_{j,m} } \big| T^m_{j,\nu,\mu} f\big |^2\Big)^{1/2} \Big\|_4^4\bigg)^{1/4} \,.
\end{multline*} 
We will pick up the term $N^{1/4}$ from the $m$-sum, and so, in order to get \eqref{eq:impr} it remains  to prove, for fixed $m\le C+N/2$, 
 \begin{equation}\label{eq:m-ind}
   \bigg(\sum_{\substack{j: \\ 2m-C\le j \le N} }  \sum_{|\ka|\le C'2^{-m+\frac j2}}
     \Big\| \Big(\sum_{\mu\in \fV^\ka_{j,m} } \sum_\nu \big| T^m_{j,\nu,\mu} f\big |^2\Big)^{1/2} \Big\|_4^4\bigg)^{1/4} \lc \|f\|_4\,,
 \end{equation}
 with the  implicit constant independent of $m$.

 Inequality \eqref{eq:m-ind}
 follows by an interpolation argument. As indicated before it is now useful to consider the operators 
 $T^m_{j,\nu,\mu}$ act as  bilinear operators  $\cB^m_{j,\mu,\nu} $ acting on $f$ and the $\om_j$.

Let  $\Omega= \{\Omega_j\}_{j=1}^N $ in $\ell^\infty(L^r_a))$  where  $L^r_a$ is the usual Sobolev (or Bessel potential) space.
Let 
$\chi_1\in C^\infty$,  $\chi_1=1$ on $(\frac 18,\frac 12)$ and $\supp(\chi_1)\subset(\frac 1{16},\frac 34)$, and define 
$\cB^m_{j,\mu,\nu} (f,\Om)$ by 
\begin{align*}  \widehat{\cB^m_{j,\mu,\nu} (f,\Om)} &
=  h_{m,j,\mu, \nu}[\Om] \widehat f \text{ where }   h_{m,j,\mu, \nu}[\Om](\xi) =\\& 
 \chi_\circ(\xi)\vartheta(2^{\frac j2}\xi_1-\nu)\vartheta(2^{\frac j2+m} \xi_1-\mu) \chi_1(2^j(1-|\xi|)) \Om_j(2^j(1-|\xi|)).
 \end{align*}  
Taking $\Omega_j$ such that $\Om_j=\om_j$ 
 we see that $T^m_{j,\nu,\mu} f =\cB^m_{j,\mu,\nu} (f,\Om)$. Inequality 
\eqref{eq:m-ind} is then the special case with $q=4$ of 
 \begin{multline}\label{eq:m-ind=q} 
\bigg(\sum_{2m-C\le j\le N }  \sum_{|\ka|\le  C' 2^{-m+\frac j2}}
     \Big\| \Big(\sum_{\mu\in \fV^\ka_{j,m} } \sum_\nu \big| \cB^m_{j,\nu,\mu} (f,\Om)\big |^2\Big)^{1/2} \Big\|_q^q\bigg)^{1/q} 
    \\ \lc  \sup_j\|\Om_j\|_{L^r_a}\|f\|_q, \quad 2< q\le \infty, \,\, a>1-\tfrac 2q, \, r= \tfrac{2q}{q-2}.
     \end{multline} 
     This inequality expresses  the $\ell^\infty(L^r_a) \times L^q \to \ell^q(L^q(\ell^2))$ boundedness of a bilinear operator which can be interpolated using the complex interpolation method.  Here one combines standard results  for interpolation of bilinear operators   \cite[\S4.4]{bergh-lofstrom} and 
     interpolation of Sobolev spaces (\cite[\S6]{bergh-lofstrom}). Here for the limiting result $q=2$ we use $L^\infty$ in place of $L^r_a$. We use the interpolation in conjunction with  a Sobolev embedding theorem. Observe that for  $2<q<\infty$, $\theta=1-2/q$ and $a>1-2/q$ the complex interpolation space $[L^\infty, L^2_{1+\eps}]_\theta$ contains $L^r_a$ with $r=\frac{2q}{q-2}$ provided that $\eps>0$ is such that $(1+\eps)(1-2/q)<a$.

For $q=2$ we aim for the  $\ell^\infty(L^\infty) \times L^2 \to \ell^2(L^2(\ell^2))$
 boundedness.
 We thus prove the
 inequality  
  \begin{multline}\label{eq:m-indL2}
   \bigg(\sum_{2m-C\le j\le N} 
   \sum_{|\ka|\le  C' 2^{-m+\frac j2}}
     \Big\| \Big(\sum_{\mu\in \fV^\ka_{j,m} } \sum_\nu \big| \cB^m_{j,\nu,\mu} (f,\Om)\big |^2\Big)^{1/2} \Big\|_2^2\bigg)^{1/2} 
    \\ \lc  \sup_j\|\Om_j\|_{L^\infty} \|f\|_2.
    \end{multline}
 This holds by  almost-orthogonality. Indeed, for fixed $m$,  the non-empty supports $\cR^m_{j,\nu, \mu}$ are contained in $A_j$ and have all finite overlap, i.e.
 \[ \sup_\xi \sum_{j,\ka,\mu,\nu} \bbone_{\cR^m_{j,\nu,\mu} } (\xi)<\infty \] 
 which implies \eqref{eq:m-indL2}.

 For $q=\infty$ we prove $\ell^\infty(L^2_{1+\eps}) \times L^\infty \to \ell^\infty(L^\infty(\ell^2))$ boundedness; i.e. we show the  inequality
 \begin{multline}\label{eq:m-indLinfty}
   \sup_{2m-C\le j\le  N}   \sup_{| \ka|\le  C' 2^{m-\frac j2}}
     \Big\| \Big(\sum_{\mu\in \fV^\ka_{j,m} } \sum_\nu \big| \cB^m_{j,\nu,\mu} (f,\Om)\big |^2\Big)^{1/2} \Big\|_\infty \\ \lc \sup_j \|\Omega_j\|_{L^2_{1+\eps} (\bbR) }\|f\|_\infty.
 \end{multline}
For fixed  $j$, $m$, $\kappa$  this becomes is  essentially an $L^\infty$ estimate for equally-spaced square functions. This type  of  $L^\infty$ bound  was first noticed by Carleson (in unpublished work), 
 see \cite{CordobaBR79}) and   \cite{RubiodeFrancia1983}. For a version close to what we need here see   \cite{seeger-BRwt}.
 
For fixed $j$, $m$, $\kappa$ the multipliers $h_{m,j,\mu,\nu}$  with $\mu \in \fV^\ka_{j,m}$ are supported in 
essentially disjoint rectangles of sidelength $2^{-j/2-m}$ and $2^{-j}$, with the longer side perpendicular to $\fn_{j,\nu,\mu}$. The angle between $\fn_{j,\nu,\mu}$ and $\fn_{j,\mu',\nu'}$ for $\mu, \mu'\in \fV^\ka_{j,m} $ is $O(2^{-\frac j2+m})$ which is comparable to the ratio of the length of the shorter side to the length of the longer side. This means that for fixed $j,m,\ka$ these rectangles have  essentially the same orientation. 
Each is contained in a rectangle with short side  of length $C 2^{-j}$ parallel to a unit vector $\fn_\ka$ and a long side of length $C 2^{-m-j/2}$ parallel to the orthogonal unit vector $\fn_\ka^\perp$. Moreover $\fn_\ka$ is normal to the circle in the angular region associated with $\fV^\ka_{j,m}$. 

Now let $\La_\ka\equiv \La_{\ka,m,j} :\bbR^2\to \bbR^2$  be the linear transformation with $\La_\ka \fn_\ka= 2^j \fn_\ka$ and $\La_\ka \fn_\ka^\perp = 2^{m+\frac j2}\fn_\ka^\perp$. Let, for $(\mu,\nu)\in \fV^\ka_{j,m}$,
\[
a_{m,j,\mu,\nu}[\om]  (\xi) = h_{m,j,\mu,\nu} [\om] (L_\ka^{-1} \xi).  \]
The multipliers $a_{m,j,\mu,\nu} [\om]$
live  in balls of bounded diameter and,  given fixed $m,\ka$ each of those balls nontrivially  intersects the sets 
$\supp(a_{m,j,\mu,\nu})$ only for a bounded number  of pairs  $(\mu,\nu)$. 
Also,  
\Be \label{eq:uniformSobolev} \sup_{(\mu,\nu)\in  \fV^\ka_{j,m}}\big\|  a_{m,j,\mu,\nu}[\om] \big \|_{L^2_s(\bbR^2)} \lc \Big (\sum_{\mu,\nu} |c_{\mu,\nu}|^2\Big)^{1/2} \sup_{\mu,\nu} \|\om\|_{L^2_s(\bbR)}, 
\Ee 
for $s=0,1,2,\dots,$ a straightforward consequence of the chain rule. 

Now to prove \eqref{eq:m-indLinfty},
we use duality in $\ell^2$ to see that for each $j,m,x$  it suffices to show that for $\eps>0$
\Be \label{eq:fixedxest}
\Big| \sum_{(\mu,\nu)\in \fV^\ka_{j,m}} c_{\mu,\nu} \cF^{-1}[a_{m,j,\mu,\nu} ] * f(x) \Big| \lc \|\{c_{\mu,\nu} \}\|_{\ell^2(\bbZ)} \|\om\|_{L^2_{1+\eps} } \|f\|_\infty.
\Ee
As in the standard proof for Bernstein's inequality 
the left hand side of \eqref{eq:fixedxest} is dominated by 
\begin{align*} 
&\Big( \int\Big|\sum_{(\mu,\nu)\in \fV^\ka_{j,m}} c_{\mu,\nu}  \cF^{-1}[a_{m,j,\mu,\nu}[\om] ] (y) \Big|^2 (1+|y|)^{2+2\eps} dy\Big)^{\frac 12} \Big(\int \frac{|f(x-y)|^2} {(1+ |y|)^{2+2\eps}} dy \Big)^{\frac 12}
\\&\lc_\eps \Big\| \sum_{(\mu,\nu)\in \fV^\ka_{j,m}} c_{\mu,\nu} a_{m,j,\mu,\nu}[\om] \Big \|_{L^2_{1+\eps}(\bbR^2)}  \|f\|_\infty
\end{align*}
and the proof of \eqref{eq:fixedxest} is concluded after noting that 
\[\Big\| \sum_{(\mu,\nu)\in \fV^\ka_{j,m}} c_{\mu,\nu} a_{m,j,\mu,\nu}[\om] \Big \|_{L^2_s(\bbR^2)} \lc_s \Big (\sum_{\mu,\nu} |c_{\mu,\nu}|^2\Big)^{1/2} \sup_{\mu,\nu} \|\om\|_{L^2_s}.
\]
This inequality in turn is verified for $s=0,1,2,\dots$  by using \eqref{eq:uniformSobolev} and the described disjointness of supports. It then follows  for all $s\ge 0$ by interpolation. This finishes the proof of \eqref{eq:fixedxest} and hence \eqref{eq:m-indLinfty}.
\qed

\section{Concluding remarks}

\subsection{Other  regularizations} 
In the definition of $h_\delta$ it was convenient to   regularize the characteristic function of an interval and then to use that in the radial variable. Alternatively,  we can also choose $\widetilde h_\delta (\xi)= \bbone_D*\upsilon_\delta$ where $\upsilon_\delta(\xi)=\delta^{-2}\upsilon (\delta^{-1}\xi)$ and $\upsilon$ is a radial $C^\infty_c$ function with $\int \upsilon (\xi) d\xi=1$. Theorem \ref{thm:main} remains valid with $\widetilde h_\delta$ in place of $h_\delta$ since $\|\widetilde h_\delta-h_\delta\|_{M_4} \lc (\log\tfrac 1\delta)^{1/4}$ for small $\delta$; this can be seen by applying either C\'ordoba's argument in \cite{CordobaBR79} or Theorem  \ref{thm:Tjthm} with $N=1$. We remark  that while C\'ordoba relies {explicitly} on an $L^2$ estimate for a Nikodym maximal function, this is avoided here (indeed already avoided in \cite{seeger-BRwt}). However it was noted in \cite{CarberyRLP} that  the Fefferman argument from \cite{FeffermanBR73}, as applied in \cite{seeger-BRwt}, \cite{CarberySeeger-QJ} and here, already implies the $L^2$ estimate for the Nikodym maximal function. 

\subsection{Remarks on higher dimensions}  In this paper we have focused  on the upper bounds {entirely} in the two-dimensional situation. 
In the complement of the  range  $p\in [\frac{2d}{d+1}, \frac{2d}{d-1}]$ only  partial results for Bochner--Riesz  and mollified  disc multipliers are  known for $d\ge 3$, and much research is currently being  done in this area.  
We refer to \cite{Stein-Williamstown}, \cite{Tao-Duke99, Tao-Indiana1998} for a discussion of  the Bochner--Riesz and related conjectures.  We 
will not review the long list of important  partial  results here and refer for the most recent progress on the Bochner--Riesz problem  to  \cite{GuoOhWangWuZhang}, in conjunction with  recent  work  on the Fourier restriction conjecture \cite{WangWu24}, and the references in those papers.  Before the  resolution of the  (non-endpoint)
Bochner--Riesz conjecture sharp results in the range $\frac{2d}{d+1}\le p\le \frac{2d}{d-1}$, analogous to Theorem \ref{thm:main}, seem currently   completely out of reach in dimension $d\ge 3$.

 \bibliographystyle{plain}

\end{document}